\documentclass[12pt,reqno,draft]{amsart}
\usepackage{fullpage}
\usepackage{amsfonts}
\usepackage{amssymb}
\usepackage{enumerate}
\usepackage{times}
\usepackage{graphicx}
\usepackage{url}
\usepackage{mathrsfs}

\newtheorem{theorem}{Theorem}[section]
\newtheorem{prop}[theorem]{Proposition}

\newtheorem{lemma}[theorem]{Lemma}
\newtheorem{conj}[theorem]{Conjecture}
{
\theoremstyle{remark}
\newtheorem*{remark}{Remark}
\newtheorem*{remarks}{Remarks}
}
{
\theoremstyle{definition}
\newtheorem{definition}[theorem]{Definition}
}

\renewcommand{\mod}[1]{{\ifmmode\text{\rm\ (mod~$#1$)}\else\discretionary{}{}{\hbox{ }}\rm(mod~$#1$)\fi}}
\newcommand{\ep}{\varepsilon}
\newcommand{\leg}[2]{\genfrac{(}{)}{}{}{#1}{#2}}

\newcommand{\C}{{\mathcal C}}
\newcommand{\E}{{\mathcal E}}
\newcommand{\A}{{\mathcal A}}
\newcommand{\B}{{\mathcal B}}

\newcommand{\F}{{\mathbb F}}

\newcommand{\Q}{{\mathbb Q}}
\newcommand{\R}{{\mathbb R}}
\newcommand{\Z}{{\mathbb Z}}

\DeclareMathAlphabet{\curly}{U}{rsfs}{m}{n}
\def\Ss{\curly{S}}
\begin{document}

\title[Averages of the number of points on elliptic curves]{Averages of the number of points on elliptic curves}

\author[Greg Martin]{Greg Martin}
\address{Department of Mathematics \\ University of British Columbia \\ Room 121, 1984 Mathematics Road \\ Vancouver, BC\\V6T 1Z2\\Canada}
\email{gerg@math.ubc.ca}
\author[Paul Pollack]{Paul Pollack}
\address{Department of Mathematics \\ University of Georgia\\Boyd Graduate Studies Research Center\\ Athens, GA 30602\\USA}
\email{pollack@uga.edu}
\author[Ethan Smith]{Ethan Smith}
\address{Department of Mathematics\\Liberty University\\1971 University Blvd\\MSC Box 710052\\ Lynchburg, VA 24502\\USA}
\keywords{elliptic curves, Koblitz conjecture, mean values of arithmetic functions}
\email{ecsmith13@liberty.edu}

\subjclass[2010]{Primary 11G05, Secondary 11N37, 11N60}
\begin{abstract} If $E$ is an elliptic curve defined over $\Q$ and $p$ is a prime of good reduction for $E$, let $E(\F_p)$ denote the set of points on the reduced curve modulo $p$. Define an arithmetic function $M_E(N)$ by setting $M_E(N):= \#\{p\colon \#E(\F_p)= N\}$. Recently, David and the third author studied the average  of $M_E(N)$ over certain ``boxes'' of elliptic curves $E$. Assuming a plausible conjecture about primes in short intervals, they showed the following: for each $N$, the average of $M_E(N)$ over a box with sufficiently large sides is $\sim \frac{K^{\ast}(N)}{\log{N}}$ for an explicitly-given function  $K^{\ast}(N)$.
	
The function $K^{\ast}(N)$ is somewhat peculiar: defined as a product over the primes dividing $N$, it resembles a multiplicative function at first glance. But further inspection reveals that it is not, and so one cannot directly investigate its properties by the usual tools of multiplicative number theory. In this paper, we overcome these difficulties and prove a number of statistical results about $K^{\ast}(N)$. For example, we determine the mean value of $K^{\ast}(N)$ over all $N$, odd $N$ and prime $N$, and we show that $K^{\ast}(N)$ has a distribution function. We also explain how our results relate to existing theorems and conjectures on the multiplicative properties of $\#E(\F_p)$, such as Koblitz's conjecture.
\end{abstract}
\maketitle

\section{Introduction}

Let $E$ be an elliptic curve defined over the field $\Q$ of rational numbers.  For the sake of concreteness, we assume
that the affine points of $E$ are given by a Weierstrass equation of the form
\begin{equation}\label{w eqn}
E: Y^2=X^3+aX+b,
\end{equation}
where $a$ and $b$ are integers satisfying the condition $-16(4a^3+27b^2)\ne 0$.  For any prime $p$ where $E$ has
good reduction, we let $E(\F_p)$ denote the group of $\F_p$-points on the reduced curve.
In~\cite{Kow}, Kowalski introduced the arithmetic function $M_E(N)$, defined by
\begin{equation*}
M_E(N)=\#\{ p \text{ prime}\colon \#E(\F_p) = N \}.
\end{equation*}
The Hasse bound~\cite{Has} implies that if $p$ is counted by $M_E(N)$, then $p$ lies between
$(\sqrt N-1)^2$ and $(\sqrt N+1)^2$.  Thus, $M_E(N)$ is a well-defined (finite) integer.

The problem of obtaining good estimates for $M_E(N)$ appears to be very difficult.
The condition imposed by Hasse's bound together with an upper bound sieve
gives the weak upper bound $M_E(N)\ll \sqrt N/\log(N+1)$ for any $N\ge 1$.
Except in the case that $E$ has complex multiplication, nothing stronger is known.
As we will explain later, the average value of $M_E(N)$ as $N$ varies over various sets of integers  is related to some important theorems and conjectures in number theory.  In~\cite{DS}, David and the third
author established an ``average value theorem'' for $M_E(N)$ as $E$ varies over a family of elliptic curves. That work was inspired by pioneering results of Fouvry and Murty \cite{FV96}, who proved an average value theorem for counts of supersingular primes.
Unfortunately, because of the restriction that all primes counted by $M_E(N)$ lie between $(\sqrt N-1)^2$ and
$(\sqrt N+1)^2$, the result of \cite{DS} is necessarily conditional upon a conjecture about the distribution of primes in short
intervals (see Conjecture~\ref{BDH} below).

The main result of~\cite{DS} introduced a strange arithmetic function, which was called
$K(N)$ because it is ``almost a constant''. In order to define $K(N)$, we recall the common notation
$\nu_p(n)$ for the exact power of $p$ that divides $n$, so that $n=\prod_p p^{\nu_p(n)}$. We also recall the
Kronecker symbol $\leg ab$, an extension of the Jacobi symbol that is defined for all integers $a$ and $b\ne0$
(see, for instance, \cite[Definition 1.4.8, page 28]{Coh}).

\begin{definition} \label{K definition}
For any positive integer $N$, we define
\[
K(N) = \prod_{p\nmid N} \bigg( 1 - \frac{\leg{N-1}{p}^2p + 1}{(p-1)^2(p+1)} \bigg) \prod_{p\mid N} \bigg( 1 - \frac1{p^{\nu_p(N)}(p-1)} \bigg).
\]
We also define $K^*(N) = K(N) N/\phi(N)$, where $\phi(N)$ is the usual Euler totient function.
\end{definition}

As we will see later, it is actually the function $K^*(N)$ that has an interesting connection
to the function $M_E(N)$. The purpose of the present work is a statistical study of the function $K^*(N)$.
Our computations will illustrate a technique for dealing with arithmetic functions that have a form similar to,
but are not exactly, multiplicative functions.  Our first main result is the computation of the average value of $K^*$, first over all $N$ and then over odd values of $N$.


\begin{theorem} \label{N odd average for K} For $x\geq 2$, we have
\[ \sum_{N \le x} K^*(N) = x + O\bigg( \frac x{\log x} \bigg) \quad\text{and}\quad \sum_{\substack{N\le x \\ N \text{ odd}}} K^*(N) = \frac x3 + O\bigg( \frac x{\log x} \bigg). \]
\end{theorem}
\noindent Thus $K^{\ast}$ has average value $1$ on all $N$, and average value $2/3$ on odd $N$.

Our second main result is the computation of the average value of $K^*$ on primes. We employ the usual notation $\pi(x)=\#\{p\le x \colon p\text{ is prime}\}$.

\begin{theorem} \label{N prime average for K} Fix $A>1$. Then for $x\geq 2$,
\begin{equation} \label{K or Kstar}
 \sum_{p \le x} K^*(p) = \tfrac{2}{3}C_2J \, \pi(x) + O_A\bigg(\frac{x}{(\log{x})^A}\bigg).
\end{equation}
Here the constants $C_2$ and $J$ are defined by
\begin{equation} \label{C2 def}
C_2 = \prod_{p>2} \bigg( 1 - \frac1{(p-1)^2} \bigg),
\end{equation}
and
\begin{equation}\label{J def}
J = \prod_{p > 2} \bigg(1+\frac{1}{(p-2)(p-1)(p+1)}\bigg).
\end{equation}
Furthermore, the asymptotic formula \eqref{K or Kstar} also holds for $\sum_{p \le x} K(p)$.
\end{theorem}
\begin{remark}
We have written $C_2$ and $J$ as two separate constants because $C_2$ arises naturally by itself in the analysis of the function $K(N)$ (see equation~\eqref{K decomposition}).
\end{remark}

The technique we use to establish Theorems~\ref{N odd average for K} and~\ref{N prime average for K}, which is dictated by the unusual Definition~\ref{K definition} for $K(N)$, is of interest in its own right: the function $K$ looks much like a multiplicative function but actually is not. One can rewrite Definition~\ref{K definition} in the following form:
\begin{equation} \label{K decomposition}
K(N) = C_2 F(N-1) G(N)
\end{equation}
where $C_2$ is the twin primes constant defined in equation~\eqref{C2 def},
\begin{equation}\label{F def}
F(n) = \prod_{\substack{p\mid n \\ p>2}} \bigg( 1 - \frac1{(p-1)^2} \bigg)^{-1} \prod_{p\mid n}\bigg( 1-\frac1{(p-1)^2(p+1)} \bigg),
\end{equation}
and
\begin{equation}\label{G def}
G(n) = \prod_{\substack{p\mid n \\ p>2}} \bigg( 1 - \frac1{(p-1)^2} \bigg)^{-1} \prod_{p^{\alpha}\parallel n} \left(1-\frac{1}{p^{\alpha}(p-1)}\right).
\end{equation}
%

So to understand the average value of $K(N)$, we are forced to deal with the correlation between the multiplicative function $F$, evaluated at $N-1$, and the multiplicative function $G$ evaluated at the neighboring integer $N$. It is perhaps somewhat surprising that the average values of $C_2 F(N-1) G(N)$ described in Theorem \ref{N odd average for K} come out to simple rational numbers.
%
%
%

The fact that we can successfully compute average values of the function $K^*$, even though it is not truly multiplicative, makes it natural to wonder whether we can analyze $K^*$ in other ways; this is indeed the case. Our next result is an analogue for $K^{*}(N)$ of a classical result of Schoenberg \cite{Sch} for the function $n/\phi(n)$. Recall that a \emph{distribution function} $D(u)$ is a nondecreasing, right-continuous function $D\colon \R\to[0,1]$ for which $\lim_{u\to-\infty} D(u)=0$ and $\lim_{u\to\infty} D(u)=1$.

\begin{theorem} \label{distribution function theorem} The function $K^*$ possesses a distribution function relative to the set of all natural numbers $N$. In other words, there exists a distribution function $D(u)$ with the property that at each of its points of continuity,
\[ D(u) = \lim_{x\to\infty} \frac{1}{x}\#\{N\leq x \colon K^{*}(N) \leq u\}. \]
\end{theorem}

As a consequence of Theorems~\ref{N odd average for K} and~\ref{N prime average for K}, we are able to show
that the main result of~\cite{DS} is consistent with various unconditional results.
As mentioned above, the restriction imposed by the Hasse bound creates a short-interval problem in any study of
$M_E(N)$ when $N$ is held fixed.  Indeed, the interval is so short that not even the Riemann hypothesis is
any help.  This problem is circumvented in~\cite{DS} by assuming a conjecture in the spirit of the
classical Barban--Davenport--Halberstam theorem.

\begin{conj} \label{BDH}
Recall the notation $\theta(x;q,a) = \sum_{p\le x,\, p\equiv a\mod q} \log p$.
Let $0<\eta\le 1$ and $\beta>0$ be real numbers. Suppose that $X$, $Y$, and $Q$ are positive real numbers satisfying $X^\eta \le Y \le X$ and $Y/(\log{X})^{\beta} \le Q \le Y$. Then
\[
\sum_{q\le Q} \sum_{\substack{1\leq a \leq q\\(a,q)=1}} \bigg| \theta(X+Y;q,a) - \theta(X;q,a) - \frac Y{\phi(q)} \bigg|^2 \ll_{\eta,\beta}
YQ\log X.
\]
\end{conj}

\begin{remark} We remark that Languasco, Perelli, and Zaccagnini~\cite{LPZ} have established Conjecture~\ref{BDH} in the range $\eta>\frac7{12}$; they also showed, assuming the generalized Riemann hypothesis, that any $\eta>\frac12$ is admissible.
\end{remark}

Given integers $a$ and $b$ satisfying $-16(4a^3+27b^2)\ne 0$, let $E_{a,b}$ denote the elliptic curve given
by the Weierstrass equation~\eqref{w eqn}.  Then, given positive parameters $A$ and $B$, let $\E(A,B)$ denote
the set defined by
\begin{equation*}
\E(A,B) = \{ E_{a,b}\colon |a|\le A,\, |b|\le B,\, -16(4a^3+27b^2)\ne0 \}
\end{equation*}
In~\cite{DS,DScor}, David and the third author established the following average value theorem (in fact a stronger version of it) for $M_E(N)$ taken over the family $\E(A,B)$.

\begin{prop} \label{DS prop}
Assume the Barban--Davenport--Halberstam estimate (Conjecture~\ref{BDH}) holds for some $\eta<\frac12$. Let $\ep$ be a positive real number, and let $A>N^{1/2+\ep}$ and $B>N^{1/2+\ep}$ be real numbers satisfying $AB > N^{3/2+\ep}$. Then for any positive real number $R$,
\[
\frac{1}{\#\E(A,B)}\sum_{E \in \E(A,B)} M_E(N) = \frac{K^*(N)}{\log N} + O_{\eta,\ep,R}\bigg(\frac{1}{(\log{N})^R}\bigg).
\]
\end{prop}
\begin{remarks}\mbox{ }
\begin{enumerate}
\item It is not necessary to assume that Conjecture~\ref{BDH} holds for a fixed $\eta<1/2$.  It is enough to assume
that it holds for $Y=\sqrt X/(\log X)^{\beta+2}$.
\item The originally published formula in~\cite{DS} contained an error in the definition of $K^*(N)$, which was corrected in~\cite{DScor} to the form given in Definition~\ref{K definition}. See the end of Section~\ref{odd section} for further discussion of the original version of $K^*(N)$.
\item The proof of Proposition~\ref{DS prop} given in~\cite{DS} is restricted to odd values of $N$, but further work by Chandee, Koukoulopoulos, David, and Smith~\cite{CDKS} establishes the proposition for even values of $N$ as well.
\end{enumerate}
\end{remarks}

We note, as in~\cite{Kow}, that computing the average value of $M_E(N)$ over the integers $N\le x$ is easily
seen to be equivalent to the prime number theorem.
In particular,
\begin{equation}\label{MEN and PNT}
\sum_{N\le x}M_E(N)=\sum_{p\le(\sqrt x+1)^2}\#\{N\le x\colon \#E(\F_p)=N\}=\pi(x)+O\left(\sqrt x\right).
\end{equation}
Similarly, the average value of $M_E(N)$ taken over the
integers $N\le x$ that satisfy a congruence condition is equivalent to an
appropriate application of the Chebotarev density theorem.  For example, if the $2$-division field of $E$ is
an $S_3$-extension of $\Q$, then the Chebotarev density theorem implies that
\begin{equation*}
\sum_{\substack{N\le x\\ N\text{ odd}}}M_E(N)\sim\frac{1}{3}\frac{x}{\log x}.
\end{equation*}
(The calculation of the constant $\frac13$ reduces to the fact that two thirds of the elements of ${\rm GL}_2(\Z/2\Z)$, which is the automorphism group of $E[2]$, have even trace.)
If $E$ is given by the Weierstrass equation~\eqref{w eqn}, the $2$-division field is easily seen to be the
splitting field of the polynomial $X^3+aX+b$.  Since almost all cubics (when ordered by height) have $S_3$ as their Galois groups, it seems reasonable to conjecture that
\begin{equation}\label{N odd double average}
\frac1{\#\E(A,B)} \sum_{\substack{N \le x \\ N \text{ odd}}} \sum_{E\in\E(A,B)} M_E(N) = \frac x{3\log x}
	+ O\bigg( \frac x{(\log x)^2} \bigg),
\end{equation}
provided that $A$ and $B$ are growing fast enough with respect to $x$. A precise version of this conjecture was established by Banks and Shparlinski \cite[Theorem 19]{BS}. (In fact, their theorem shows that an analogous estimate holds with the condition ``$N$ odd'' replaced by ``$m \nmid N$'', for any given integer $m$.) The asymptotic result \eqref{N odd double average}, together with the result of Theorem~\ref{N odd average for K} for odd $N$, shows that if we average the two sides of the equation in Proposition~\ref{DS prop}, we obtain consistent results (unconditionally).
Similarly, the result of Theorem~\ref{N odd average for K} for all $N$ allows us to infer the asymptotic formula
\begin{equation*}
\frac1{\#\E(A,B)} \sum_{N \le x} \sum_{E\in\E(A,B)} M_E(N) = \frac x{\log x} + O\bigg( \frac x{(\log x)^2} \bigg),
\end{equation*}
which is consistent with equation~\eqref{MEN and PNT}.
We can therefore, if we wish, view Theorem~\ref{N odd average for K} as additional evidence for the conclusion of Proposition~\ref{DS prop}.

A similar problem arises if we consider only primes $p$.
Computing the average value of $M_E(p)$ over the primes $p\le x$ is easily seen to be
equivalent to the famous Koblitz conjecture~\cite{Kob}:

\begin{conj}[Koblitz] \label{Koblitz}
Given an elliptic curve $E$ defined over the rational field $\Q$, there exists a constant $C(E)$ with the property that as $x\to\infty$,
\[
\sum_{\substack{p \le x \\ p \text{ prime}}} M_E(p) \sim C(E) \frac x{(\log x)^2}.
\]
\end{conj}
The constant $C(E)$ appearing in Koblitz's conjecture may be zero, in which case the asymptotic is interpreted
to mean that there are only finitely many primes $p$ such that $M_E(p)>0$.
An obvious obstruction to there being infinitely many primes with $M_E(p)>0$ is for $E$ to be isogenous to a
curve possessing nontrivial rational torsion. It was once thought that this was the only case when $C(E)=0$, but this turned out to be false; see \cite[Section 1.1]{Z} for an explicit counterexample due to Nathan Jones.

The main theorem of~\cite{BCD} may be reinterpreted to say that the asymptotic formula
\begin{align} \label{N prime double average}
\frac1{\#\E(A,B)} \sum_{\substack{p \le x \\ p \text{ prime}}} \sum_{E\in\E(A,B)} M_E(p)
&= \tfrac23C_2J \int_2^x \frac{dt}{(\log t)^2} + O_A\bigg( \frac x{(\log x)^A} \bigg) \\
&= \tfrac23C_2J \frac x{(\log x)^2} + O\bigg( \frac x{(\log x)^3} \bigg)  \notag
\end{align}
holds unconditionally for $A$ and $B$ growing fast enough with respect to $x$.
Jones~\cite{J} has averaged the explicit formula for $C(E)$ over the family $\E(A,B)$ and shown that the result is consistent with the above formula.  We view this as providing good evidence for the Koblitz conjecture.  Equation~\eqref{N prime double average}, together with our Theorem~\ref{N prime average for K}, shows that we obtain consistent results (unconditionally) when we average the two sides of the equation in Proposition~\ref{DS prop} over the primes $N\le x$.
Thus all of the conjectures and conditional theorems mentioned above reinforce one another's validity.

We note that the asymptotic formulas~\eqref{N odd double average} and~\eqref{N prime double average}, in which we average over odd integers $N$ or primes $p$ up to $x$, both hold for a much wider range of $A$ and $B$ than is suggested by Proposition~\ref{DS prop}.  In particular, Banks and Shparlinski~\cite{BS} developed a character-sum
argument based on a large sieve inequality to show that one may take $A, B>x^{\epsilon}$ and $AB>x^{1+\epsilon}$ in elliptic-curve averaging problems of this sort, when the average number of elliptic curve isomorphism classes modulo $p$ satisfying the desired property is somewhat large.
Baier~\cite{Bai} was able to adapt this technique to make similar improvements to the required length of the average in the (fixed trace) Lang--Trotter problem, where the average number of classes modulo $p$ is significantly smaller. Given Baier's result, it seems possible 
that Proposition~\ref{DS prop}, in which the odd integer $N$ is fixed, could itself be shown to hold provided that $A, B>N^{\epsilon}$ (note that such an improvement would still seem to require that $AB>N^{3/2+\epsilon}$ rather than the weaker condition $AB>N^{1+\epsilon}$). As we are primarily concerned with the multiplicative function $K^*$ herein, however, we have not pursued this line of thinking.

The remainder of the article is organized as follows. We begin by establishing Theorem~\ref{N odd average for K} in Section~\ref{odd section}. Briefly, we approximate the function $K^*(N)$ by a similar function whose values depend only upon the small primes dividing $N$ and $N-1$; we then calculate the average value of this truncated function by partitioning the numbers being averaged over into ``configurations'' based on local data about $N$ and $N-1$ at these small primes. We prove the related Theorem~\ref{N prime average for K} in Section~\ref{prime section}; here the calculation of the main term is simpler since the argument of $K^*$ is always a prime, while the estimation of the error term is more complicated due to the need to invoke results on the distribution of primes in arithmetic progressions. Finally, we establish Theorem~\ref{distribution function theorem} in Section~\ref{distribution section} by studying the moments of~$K^{*}$.

\subsection*{Notation} As above, we employ the Landau--Bachmann $o$ and $O$ notation, as well as the associated Vinogradov symbols $\ll$, $\gg$ with their usual meanings; any dependence of implied constants on other parameters is denoted with subscripts. We reserve the letters $\ell$ and $p$ for prime variables. For each natural number $n$, we let $P(n)$ denote the largest prime factor of $n$, with the convention that $P(1)=1$. The natural number $n$ is said to be \emph{$y$-friable} (sometimes called \emph{$y$-smooth}) if $P(n)\leq y$. We write $\Psi(x,y)$ for the number of $y$-friable integers not exceeding~$x$.
By a \emph{partition} of a set $S$, we mean any collection of disjoint sets whose union is $S$; we do \emph{not} require that all of the sets in the collection be nonempty.

\section{The average value of $K^*$} \label{odd section}
For notational convenience, set $R(N) := N/\phi(N)$, so that $K^{*}(N) = K(N)R(N)$. By definition, $K(N)$ is a product over primes, while $R(N) = \prod_{\ell \mid N}(1-1/\ell)^{-1}$ can also be viewed as such a product. Moreover, it is the small primes that have the largest influence on the magnitude of these products.  This suggests it might be useful to study the truncated functions $K_z$ and $R_z$ defined by
\[ K_z(N):= \prod_{\substack{p\nmid N\\ p \leq z}} \bigg( 1 - \frac{\leg{N-1}{p}^2p + 1}{(p-1)^2(p+1)} \bigg) \prod_{\substack{p\mid N \\ p \leq z}} \bigg( 1 - \frac1{p^{\nu_p(N)}(p-1)} \bigg),\]
and
\[ R_z(N):= \prod_{\substack{p \mid N \\ p \leq z}} \left(1-1/p\right)^{-1}. \]

We give the proof of the first half of Theorem \ref{N odd average for K}, concerning the average of $K(N) R(N)$ over all $N$, in complete detail. The proof of the second claim, concerning the average over odd $N$, can be proved in the same way; the necessary changes to the argument are indicated briefly at the end of this section.

The first half of Theorem \ref{N odd average for K} will be deduced from a corresponding estimate for the mean value of $K_z(N) R_z(N)$:

\begin{prop}\label{prop:oddavg} Let $x \geq 3$, and set $z := \frac{1}{10}\log{x}$. We have
\[ \sum_{N \leq x} K_z(N) R_z(N) = x + O(x^{3/4}). \]
\end{prop}

We will establish this proposition at the end of this section (it follows upon combining Lemmas~\ref{lem:related} and~\ref{lem:identity}). At this point, we show how Theorem~\ref{N odd average for K} can be deduced from the proposition.

\begin{proof}[Proof of Theorem~\ref{N odd average for K}, assuming Proposition~\ref{prop:oddavg}] It suffices to show that with $z = \frac{1}{10}\log{x}$,
	\begin{equation} \label{glitch fixer}
	\sum_{\substack{N\le x \\ N\text{ odd}}} \big| K_z(N) R_z(N) - K(N) R(N) \big| \ll x/z.
	\end{equation}
Now $0 \leq K(N) \leq K_z(N) \leq 1$ and $0 \leq R_z(N) \leq R(N)$, so that
\begin{align*} |K_z(N) R_z(N) - K(N) R(N)| &\leq |K_z(N)||R_z(N)-R(N)| + |K_z(N)-K(N)| R(N)\\
	&\leq  (R(N)-R_z(N)) + (K_z(N)-K(N))R(N). \end{align*}
Thus, it is enough to show that the sums up to $x$ of $R(N)-R_z(N)$ and $(K_z(N)-K(N))R(N)$ are also $\ll x/z$. As we are looking only for upper bounds, we may extend these sums over all $N \leq x$ and not only odd~$N$.

Write $R(N)= \sum_{d \mid n} g(d)$ for an auxiliary function $g$. By a straightforward calculation with the M\"{o}bius inversion formula, we see that $g$ vanishes except at squarefree integers $d$, in which case $g(d) = 1/\phi(d)$. Hence, for all real $t > 0$,
\begin{align} \sum_{N \leq t} R(N) = \sum_{N \leq t} \sum_{d \mid N} g(d) &= \sum_{d\le t} \frac1{\phi(d)} \sum_{\substack{N \leq t \\ d \mid N}} 1 \notag \\
&\le \sum_{d \leq t} \frac{t}{d\phi(d)} \notag \\
&\le t \sum_{d=1}^\infty \frac{1}{d\phi(d)} \notag \\
&= t \prod_{p}\left(1 + \frac{1}{p(p-1)} + \frac{1}{p^3(p-1)} + \dots \right) \ll t, \label{eq:onaverage}
\end{align}
so that $R(N)$ is bounded on average. Now writing $R_z(N)=\sum_{d\mid n}g_z(d)$ for an auxiliary function $g_z(d)$, one finds that $g_z$ vanishes except on squarefree $z$-friable integers $d$, in which case again $g_z(d)= 1/\phi(d)$. In particular, $g(d)-g_z(d)$ is nonnegative for all $d$, and $g(d)-g_z(d)=0$ when $d\le z$. We deduce that
\begin{align*}
\sum_{N \leq x} (R(N)-R_z(N)) = \sum_{N \leq x} \sum_{d \mid N} (g(d)-g_z(d)) &\leq \sum_{N \leq x} \sum_{\substack{d \mid N\\ d > z}} \frac{1}{\phi(d)} \\
&= \sum_{z<d\le x} \sum_{\substack{N\le x \\ d \mid N}} \frac{1}{\phi(d)} \le \sum_{d > z}\frac{x}{d\phi(d)}.
\end{align*}
Partitioning this last sum into dyadic intervals, we have
\begin{align*}
\sum_{N \leq x} (R(N)-R_z(N)) \le \sum_{k=1}^\infty \sum_{2^{k-1}z < d \le 2^kz} \frac{x}{d\phi(d)} &= x \sum_{k=1}^\infty \sum_{2^{k-1}z < d \le 2^kz} \frac{R(d)}{d^2} \\
&\le x \sum_{k=1}^\infty \frac1{(2^{k-1}z)^2} \sum_{d \le 2^kz} R(d) \\
&\ll x \sum_{k=1}^\infty \frac1{(2^{k-1}z)^2} 2^k z \\
&\ll \frac xz \sum_{k=1}^\infty \frac1{2^k} \ll \frac xz,
\end{align*}
where we used the estimate \eqref{eq:onaverage} in the second-to-last inequality. This proves the desired upper bound for the partial sums of $R(N)-R_z(N)$.

The partial sums of $(K_z(N)-K(N))R(N)$ are easier. Since each factor appearing in the products defining $K_z$ and $K$ has the form $1 - O(1/\ell^2)$, it follows that
$K(N)/K_z(N) \geq 1-O\left(\sum_{\ell > z} 1/\ell^2\right) \geq 1 - O(1/z)$. Thus, $K_z(N)-K(N) = K_z(N)(1-K(N)/K_z(N)) \leq 1-K(N)/K_z(N) \ll 1/z$. It follows that
\[ \sum_{N \leq x} (K_z(N)-K(N)) R(N) \ll \frac{1}{z}\sum_{N \leq x}R(N) \ll \frac xz,\]
using the estimate~\eqref{eq:onaverage} once more in the last step. This completes the proof of Theorem~\ref{N odd average for K}, assuming Proposition~\ref{prop:oddavg}.
\end{proof}

In the remainder of this section, we concentrate on proving Proposition \ref{prop:oddavg}. Our strategy, already alluded to in the introduction, is to partition the integers $N \leq x$ according to local data at small primes. We choose the partition so that the values $K_z(N)$ and $R_z(N)$ are constant along each set belonging to the partition (which we call a \emph{configuration}). For the remainder of this section, we continue to assume that $x\geq 3$ and that $z = \frac{1}{10}\log{x}$.

\begin{definition} \label{configuration space def}
We define the \emph{configuration space} $\Ss$ as the set of all $4$-tuples of the form \[ (\A, \B, \C, \{e_{\ell}\}_{\ell \in \B}),\] where the sets $\A, \B, \C$ partition the set of primes up to $z$, and the $e_{\ell}$ are positive integers. (Although $\Ss$ depends upon $z$ and hence $x$, we will not include this dependence in the notation.)
\end{definition}

To each $N \leq x$, we can associate a unique configuration in the following manner.

\begin{definition}\label{def:correspondingconfig} Given $N \leq x$, define three subsets of the primes in $[2,z]$ by setting $\A:= \{\ell \leq z: \ell \nmid N(N-1)\}$, $\B:=\{\ell \leq z: \ell \mid N\}$, and $\C:=\{\ell \leq z: \ell \mid N-1\}$. For each $\ell \in \B$, set $e_{\ell} := \nu_{\ell}(N)$.
Then $\sigma = (\A, \B, \C, \{e_{\ell}\}_{\ell \in \B}) \in \Ss$ is called the \emph{configuration} $\sigma$ corresponding to $N$ and is denoted $\sigma_{N}$.
\end{definition}

\begin{remark} One checks easily that the value $K_z(N) R_z(N)$ depends only on $\sigma=\sigma_{N}$. Thus, we often abuse notation by referring to $K_z(\sigma)$ and $R_z(\sigma)$ instead of $K_z(N)$ and $R_z(N)$. \end{remark}
	
We can rewrite the sum considered in Proposition \ref{prop:oddavg} in the form
\begin{equation}\label{eq:rewrite0} \sum_{N \leq x} K_z(N) R_z(N) = \sum_{\sigma \in \Ss} K_z(\sigma) R_z(\sigma)\sum_{\substack{N \leq x \\ \sigma_N = \sigma}}1. \end{equation}
In the next lemma, we estimate the inner sum on the right-hand side of \eqref{eq:rewrite0} in two ways.

\begin{lemma}\label{lem:innersum} For each $\sigma \in \Ss$, we have
\begin{equation}\label{eq:lesscrude} \sum_{\substack{N \leq x \\ \sigma_N = \sigma}}1 = d_{\sigma} x + O(x^{1/5}),\end{equation}
where
\begin{equation}\label{eq:dsigmadef} d_{\sigma}:= \bigg(\prod_{\ell \in \A} (1-2/\ell) \bigg) \bigg(\prod_{\ell \in \B}\frac{1}{\ell^{e_\ell}}(1-1/\ell)\bigg) \bigg(\prod_{\ell \in \C}\frac{1}{\ell}\bigg).\end{equation}
We also have the crude upper bound
\begin{equation}\label{eq:crudeupper}\sum_{\substack{N \leq x \\ \sigma_N = \sigma}}1 \leq x \prod_{\ell \in \B}\ell^{-e_\ell}\end{equation}
for any $\sigma\in\Ss$.
\end{lemma}
\begin{proof} The condition that $\sigma_{N}=\sigma$ is equivalent to a congruence condition on $N$ modulo \begin{equation}\label{eq:msigmadef} m_{\sigma}:= \left(\prod_{\ell \in \A \cup \C} \ell\right) \left(\prod_{\ell \in \B} \ell^{e_\ell+1}\right). \end{equation} Indeed, $\sigma_{N}=\sigma$ precisely when $N$ belongs to a union of $\prod_{\ell \in \A}(\ell-2) \prod_{\ell \in \B}(\ell-1)$ congruence classes modulo $m_{\sigma}$. This implies that
\[
\sum_{\substack{N \leq x\\\sigma_N = \sigma}} 1 = \frac x{m_\sigma} \prod_{\ell \in \A}(\ell-2) \prod_{\ell \in \B}(\ell-1) + O\bigg(\prod_{\ell \in \A \cup \B} \ell \bigg) = d_{\sigma} x + O\bigg(\prod_{\ell \leq z} \ell \bigg).
\]
By our choice of $z$ and the prime number theorem, $\prod_{\ell \leq z} \ell < x^{1/5}$ for large~$x$, and so we have established the formula~\eqref{eq:lesscrude}. To justify the inequality~\eqref{eq:crudeupper}, it suffices to observe that if $\sigma_N=\sigma$, then $\prod_{\ell \in \B} \ell^{e_\ell}$ divides~$N$.
\end{proof}

The modulus $m_{\sigma}$, defined in \eqref{eq:msigmadef}, will continue to play a key role in subsequent arguments. It will be convenient to know that $m_{\sigma}$ nearly determines $\sigma$; this is the substance of our next result.

\begin{lemma}\label{lem:nearlydetermined} For each natural number $m$, the number of $\sigma \in \Ss$ with $m_{\sigma}=m$ is $O(x^{1/4})$.
\end{lemma}
\begin{proof} Suppose that $m_{\sigma}=m$, where $\sigma=(\A, \B, \C, \{e_{\ell}\}_{\ell \in \B})$. Since the sets $\A, \B, \C$ partition the primes up to $z$, the number of possibilities for these sets is $3^{\pi(z)} = \exp(O(\log{x}/{\log\log{x}})) = x^{o(1)}$. Having chosen these sets, the exponents $e_\ell$, for $\ell \in \B$, are determined by the prime factorization of $m$. This proves the lemma with $\frac14$ replaced by any positive $\epsilon$.
\end{proof}
	
We next investigate two sums over~$m_\sigma$ for future use in estimating error terms.

\begin{lemma} \label{msigma lemma}
For each $\sigma \in \Ss$, define $m_{\sigma}$ by \eqref{eq:msigmadef}. Then for all $x\ge3$,
\begin{equation} \label{two Rankin sums}
x^{6/5} \log\log{x} \sum_{\substack{\sigma \in \Ss \\ m_{\sigma} > x}} \frac{1}{m_{\sigma}} + x^{1/5}\log\log{x} \sum_{\substack{\sigma \in \Ss \\ m_{\sigma} \le x}} 1 \ll x^{3/4}.
\end{equation}
\end{lemma}

\begin{proof}
We proceed by Rankin's method:
\begin{align*}
x^{6/5} \log\log{x} & \sum_{\substack{\sigma \in \Ss\\ m_{\sigma} > x}} \frac{1}{m_{\sigma}} + x^{1/5}\log\log{x} \sum_{\substack{\sigma \in \Ss \\ m_{\sigma} \le x}} 1 \\
&\le x^{6/5} \log\log{x} \sum_{\substack{\sigma \in \Ss \\  m_{\sigma} > x}} \bigg(\frac{m_\sigma}x\bigg)^{7/8} \frac{1}{m_{\sigma}} + x^{1/5} \log\log{x} \sum_{\substack{\sigma \in \Ss  \\ m_{\sigma} \le x}} \bigg(\frac x{m_\sigma}\bigg)^{1/8} \\
&= x^{13/40} \log\log x \sum_{\sigma \in \Ss} \frac1{m_\sigma^{1/8}}.
\end{align*}
Every value of $m_\sigma$ is $z$-friable, and there are at most $x^{1/4}$ configurations $\sigma\in\Ss$ for every possible value of $m_\sigma$ by Lemma~\ref{lem:nearlydetermined}. Therefore
\begin{align*}
x^{13/40} \log\log x \sum_{\sigma \in \Ss} \frac1{m_\sigma^{1/8}} &\ll x^{13/40} \log\log x \cdot x^{1/4} \sum_{m\ z\text{-friable}} \frac1{m^{1/8}} \\
&= x^{23/40} \log\log x \prod_{p\le z} \bigg( 1 + \frac1{p^{1/8}} + \frac1{p^{1/4}} + \cdots \bigg) \\
&= x^{23/40} \log\log x \prod_{p\le z} \bigg( 1 - \frac1{p^{1/8}}\bigg)^{-1}.
\end{align*}
Each factor in the product is at most $( 1 - 2^{-1/8})^{-1} < 13$, and so the product is less than $13^{\pi(z)} = 13^{O(\log x/\log\log x)} = x^{o(1)}$. Thus the left-hand side of equation~\eqref{two Rankin sums} is $\ll x^{23/40+o(1)} \log\log x \ll x^{3/4}$ as claimed.
\end{proof}

The next lemma relates the mean value of $K_z(N) R_z(N)$, taken over odd $N$, to the sum of $K_z(\sigma) R_z(\sigma) d_{\sigma}$, taken over all configurations $\sigma$.

\begin{lemma}\label{lem:related} For all $x\ge 3$,
	\[ \sum_{N\le x} K_z(N) R_z(N) = x \sum_{\sigma \in \Ss} K_z(\sigma) R_z(\sigma) d_{\sigma} + O(x^{3/4}). \]
\end{lemma}

\begin{proof}
We begin by noting that the upper bounds
\begin{equation} \label{easy upper bounds}
0 \le K(N) \le K_z(N) \le 1 \quad\text{and}\quad 0 \le R_z(N) \le R(N) \le \prod_{p\le x} \bigg( 1-\frac1p \bigg)^{-1} \ll \log\log x
\end{equation}
are valid for all $N\le x$. We write
\begin{align*}
\sum_{N \leq x} K_z(N) R_z(N) &= \sum_{\sigma \in \Ss} K_z(\sigma) R_z(\sigma)\sum_{\substack{N \leq x \\ \sigma_N = \sigma}} 1 \\
&= \sum_{\substack{\sigma \in \Ss \\ m_\sigma \le x}} K_z(\sigma) R_z(\sigma)\sum_{\substack{N \leq x \\ \sigma_N = \sigma}} 1 + \sum_{\substack{\sigma \in \Ss  \\ m_\sigma > x}} K_z(\sigma) R_z(\sigma)\sum_{\substack{N \leq x \\ \sigma_N = \sigma}} 1 \\
&= \sum_{\substack{\sigma \in \Ss \\ m_\sigma \le x}} K_z(\sigma) R_z(\sigma) (d_{\sigma}x + O(x^{1/5})) + O\bigg( \sum_{\substack{\sigma \in \Ss  \\ m_\sigma > x}} K_z(\sigma) R_z(\sigma) x\prod_{\ell \in \B} \ell^{-e_\ell} \bigg)
\end{align*}
by Lemma~\ref{lem:innersum}. Using the upper bounds~\eqref{easy upper bounds} for $K_z$ and $R_z$, we deduce after extending the first sum to infinity that
\begin{align*}
\sum_{N \leq x} K_z(N) R_z(N) &= x \sum_{\sigma \in \Ss} K_z(\sigma) R_z(\sigma) d_{\sigma} + O\bigg( x\log\log x \sum_{\substack{\sigma \in \Ss\\ m_\sigma > x}} d_{\sigma} \bigg) \\
&\qquad{}+ O\bigg( x^{1/5} \log\log x \sum_{\substack{\sigma \in \Ss \\ m_\sigma \le x}} 1 + x \log\log x \sum_{\substack{\sigma \in \Ss \\ m_\sigma > x}} \prod_{\ell \in \B} \ell^{-e_\ell} \bigg);
\end{align*}
since the inequality $d_\sigma \le \prod_{\ell \in \B}\ell^{-e_\ell}$ follows from the definition~\eqref{eq:dsigmadef}, the first error term is dominated by the second. Because $\prod_{\ell \in \B}\ell^{-e_\ell} = m_{\sigma}^{-1} \prod_{\ell \leq z} \ell < m_{\sigma}^{-1} x^{1/5}$ once $x$ is large, this error term is $\ll x^{3/4}$ by Lemma~\ref{msigma lemma}, and the proof is complete.
\end{proof}

In view of Lemma \ref{lem:related}, Proposition \ref{prop:oddavg} is a consequence of the following remarkable identity:
\begin{lemma}\label{lem:identity} We have
\[ \sum_{\sigma \in \Ss} K_z(\sigma) R_z(\sigma) d_{\sigma} =1. \]
\end{lemma}

\begin{proof} Referring back to the definitions of $K_z$ and $R_z$, we see that for $\sigma \in \Ss$,
\begin{multline}\label{eq:convenienteulerfactorization} K_z(\sigma) R_z(\sigma) =\left(\prod_{\ell \in \A}\left(1 - \frac{1}{(\ell-1)^2}\right)\right) \left(\prod_{\ell \in \B}\left(1-\frac{1}{\ell^{e_\ell}(\ell-1)}\right)\left(1-\frac{1}{\ell}\right)^{-1}\right) \times \\
\left(
\prod_{\ell \in \C}\left(1-\frac{1}{(\ell-1)^2(\ell+1)}\right)\right).
\end{multline}
Multiplying by the expression  \eqref{eq:dsigmadef} for $d_{\sigma}$, we find that
\begin{equation}\label{eq:bigmess} K_z(\sigma) R_z(\sigma) d_{\sigma} = \left(\prod_{\ell \in \A}\frac{\ell-2}{\ell-1}\right)^2 \left(\prod_{\ell \in \B} \frac{1}{\ell^{e_\ell}}\left(1-\frac{1}{\ell^{e_\ell}(\ell-1)}\right)\right) \left(\prod_{\ell \in \C} \frac{\ell^2-\ell-1}{(\ell-1)^2 (\ell+1)}\right).
\end{equation}
Recall that $\sigma$ is a $4$-tuple with entries $\A, \B, \C$, and $\{e_\ell\}_{\ell \in \B}$. 
We sum the expression \eqref{eq:bigmess} over the possibilities for $\{e_\ell\}$. We have
\[ \sum_{\substack{\{e_{\ell}\} \\ \text{each } e_\ell \ge 1}} \bigg(\prod_{\ell \in \B}
\frac{1}{\ell^{e_\ell}}\bigg(1-\frac{1}{\ell^{e_\ell}(\ell-1)}\bigg)\bigg) =
\prod_{\ell \in \B} \bigg(\sum_{e_\ell=1}^{\infty}\frac{1}{\ell^{e_{\ell}}}\bigg(1-\frac{1}{\ell^{e_\ell}(\ell-1)}\bigg)\bigg).
 \]
By a short computation,
\[ \sum_{e_\ell=1}^{\infty}\frac{1}{\ell^{e_{\ell}}}\left(1-\frac{1}{\ell^{e_\ell}(\ell-1)}\right) = \frac{\ell^2-2}{(\ell+1)(\ell-1)^2}. \]
Thus, if we now fix only $\A$, $\B$, and $\C$ and sum over all corresponding configurations $\sigma$, we have
\begin{align}
\sum_{\substack{\sigma\in\Ss \\ \A,\B,\C \text{ fixed}}} K_z(\sigma) R_z(\sigma) d_\sigma &= \bigg(\prod_{\ell \in \A}\frac{\ell-2}{\ell-1}\bigg)^2 \bigg(\prod_{\ell \in \B} \frac{\ell^2-2}{(\ell+1)(\ell-1)^2}\bigg) \bigg(\prod_{\ell \in \C} \frac{\ell^2-\ell-1}{(\ell-1)^2 (\ell+1)}\bigg) \notag \\
&= \bigg(\prod_{\ell \in \A} P_\A(\ell)\bigg) \bigg(\prod_{\ell \in \B}P_\B(\ell)\bigg) \bigg(\prod_{\ell \in \C}P_\C(\ell)\bigg), \label{eq:nearfinal}
\end{align}
where for notational convenience we have defined
\begin{equation} \label{early P defs}
 P_\A(\ell) = \bigg(\frac{\ell-2}{\ell-1}\bigg)^2, \quad P_\B(\ell)= \frac{\ell^2-2}{(\ell+1)(\ell-1)^2}, \quad P_\C(\ell) =  \frac{\ell^2-\ell-1}{(\ell-1)^2 (\ell+1)}.
 \end{equation}
To finish the proof, we sum the right-hand side of equation~\eqref{eq:nearfinal} over all possibilities for $\A$, $\B$, and~$\C$. The only condition on the sets $\A$, $\B$, and $\C$ is that they partition the set of primes not exceeding $z$. Hence,
\begin{align*}
\sum_{\sigma \in \Ss} K_z(\sigma) R_z(\sigma) d_{\sigma} &= \sum_{\substack{\A, \B, \C \text{ disjoint} \\ \A\cup\B\cup\C = \{\ell\le z\}}} \bigg(\prod_{\ell \in \A} P_\A(\ell)\bigg) \bigg(\prod_{\ell \in \B}P_\B(\ell)\bigg) \bigg(\prod_{\ell \in \C}P_\C(\ell)\bigg) \\
&= \prod_{\ell \leq z}\big(P_\A(\ell) + P_\B(\ell) + P_\C(\ell)\big).
\end{align*}
However, $P_\A(\ell) + P_\B(\ell)+ P_\C(\ell)=1$, identically! This completes the proof of the lemma, and so also of Proposition~\ref{prop:oddavg}.\end{proof}

As already remarked above, the first half of Theorem \ref{N odd average for K} follows immediately upon combining  Lemmas~\ref{lem:related} and~\ref{lem:identity}.

\begin{proof}[Proof of the second half of Theorem \ref{N odd average for K}] The condition that $N$ is odd amounts to the requirement that $2 \in \C$ in the configuration notation of this section. If we carry this requirement through the proofs of Lemmas~\ref{lem:related} and~\ref{lem:identity}, the bulk of the argument is essentially unchanged, but the new conclusions are that
\[
\sum_{\substack{N\le x \\ 2\nmid N}} K_z(N) R_z(N) = x \sum_{\substack{\sigma \in \Ss \\ 2\in\C}} K_z(\sigma) R_z(\sigma) d_{\sigma} + O(x^{3/4})
\]
and
\begin{align*}
\sum_{\substack{\sigma \in \Ss\\ 2 \in \C}} K_z(\sigma) R_z(\sigma) d_{\sigma} &= \sum_{\substack{\A, \B, \C \text{ disjoint} \\ \A\cup\B\cup\C = \{\ell\le z\} \\ 2 \in \C}} \bigg(\prod_{\ell \in \A} P_\A(\ell)\bigg) \bigg(\prod_{\ell \in \B}P_\B(\ell)\bigg) \bigg(\prod_{\ell \in \C}P_\C(\ell)\bigg) \\
&= P_C(2) \prod_{2< \ell \leq z}\big(P_\A(\ell) + P_\B(\ell) + P_\C(\ell)\big) = P_C(2).
\end{align*}
(We assume in going from the first line to the second that $z\geq 2$, i.e., that $x \geq e^{20}$.)
Since $P_C(2)=\frac{1}{3}$, the second half of Theorem \ref{N odd average for K} follows.
\end{proof}

Most mathematical coincidences have explanations, of course, and the magical-seeming $P_\A(\ell) + P_\B(\ell)+ P_\C(\ell)=1$ is no different. One might guess that $P_\A(\ell)$, $P_\B(\ell)$, and $P_\C(\ell)$ are probabilities of certain events occurring, and this is exactly right: as $\gamma$ ranges over all elements of ${\rm GL}_2(\F_\ell)$, the expression $\det(\gamma) + 1 - \mathop{\rm tr}(\gamma)$ is congruent to $0\mod\ell$ with probability $P_\B(\ell)$, congruent to $1\mod\ell$ with probability $P_\C(\ell)$, and congruent to each of the $\ell-2$ other residue classes with probability $P_\A(\ell)/(\ell-2)$. (See \cite[equation (2.2)]{DW} for this computation, as well as for the precise connection to elliptic curves.)

We conclude this section by saying a few words about the function that was originally published in~\cite{DS}, which we will here call $K^\circ$ to avoid confusion with the corrected function $K^*$:
\begin{multline*}
K^\circ(N) = \\
 \frac N{\phi(N)} \prod_{p\nmid N} \bigg( 1 - \frac{\leg{N-1}{p}^2p + 1}{(p-1)^2(p+1)} \bigg) \prod_{\substack{p\mid N \\ 2\nmid \nu_p(N)}} \bigg( 1 - \frac1{p^{\nu_p(N)}(p-1)} \bigg) \prod_{\substack{p\mid N \\ 2\mid \nu_p(N)}} \bigg( 1 - \frac{p-\leg{-N_p}p}{p^{\nu_p(N)+1}(p-1)} \bigg),
\end{multline*}
where $N_p = N/p^{\nu_p(N)}$ is the $p$-free part of~$N$. This function is even further from being a multiplicative function than $K^*$, since its value can depend even on the residue class modulo $p$ of the $p$-free part of~$N$. Nevertheless, our techniques can in fact determine the average value of the function $K^\circ$ as well.

To investigate the average of $K^\circ$, we would expand the notion of a configuration to a sextuple $(\A,\B_1,\B_2,\C,\{e_\ell\}_{\ell\in\B_1\cup\B_2},\{a_\ell\}_{\ell\in\B_2})$, where $\A,\B_1,\B_2,\C$ partition the set of primes up to $z$, the $e_\ell$ are positive integers, and the $a_\ell$ are integers satisfying $1\le a_\ell\le \ell-1$. We would modify Definition~\ref{def:correspondingconfig} by setting $\B_1:=\{\ell \leq z: 2\nmid e_\ell\}$ and $\B_2:=\{\ell \leq z: 2\mid e_\ell\}$ and, for $\ell\in\B_2$, choosing $a_\ell\in\{1,\dots,\ell-1\}$ so that $a_\ell \equiv N/\ell^{e_\ell} \mod \ell$. The analogue of equation~\eqref{eq:bigmess} would be
\begin{multline*}
K^\circ_z(\sigma) d_{\sigma} = \left(\prod_{\ell \in \A}\frac{\ell-2}{\ell-1}\right)^2  \left(\prod_{\ell \in \C} \frac{\ell^2-\ell-1}{(\ell-1)^2 (\ell+1)}\right) \times \\
\left(\prod_{\ell \in \B_1} \frac{1}{\ell^{e_\ell}}\left(1-\frac{1}{\ell^{e_\ell}(\ell-1)}\right)\right) \left(\prod_{\ell \in \B_2} \frac{1}{\ell^{e_\ell}(\ell-1)}\left(1-\frac{\ell-\leg{-a_\ell}\ell}{\ell^{e_\ell+1}(\ell-1)}\right)\right).
\end{multline*}
We would then hold $\A,\B_1,\B_2,\C$, and the $e_\ell$ fixed and sum over all $\prod_{\ell\in\B_2} (\ell-1)$ possibilities for the $a_\ell$; this has the effect of replacing the Legendre symbol $\leg{-a_\ell}\ell$ by its average value~$0$. At this point in the argument, the factors corresponding to primes in $\B_1$ and $\B_2$ would be identical, and the calculation would soon dovetail with equation~\eqref{eq:nearfinal}.

We felt these few details of the determination of the average value of $K^\circ$ were worth mentioning, as an example of the wider applicability of our method and the more complicated configuration spaces that can be used.

\section{The average of $K^*$ over primes} \label{prime section}

In this section we establish Theorem~\ref{N prime average for K}. The main component of the proof is the following asymptotic formula for the sum of the multiplicative function $F$ evaluated on shifted primes.

\begin{prop} \label{prop:F average}
Let $F$ be the multiplicative function defined in equation~\eqref{F def}, and let $J$ be the constant defined in equation~\eqref{J def}. For any $x>2$ and for any positive real number $A$,
\[ \sum_{p \leq x} F(p-1)  = J \pi(x) + O_A(x/(\log{x})^A). \]
\end{prop}

\begin{proof}
Write $F(n) = \sum_{d \mid n} g(d)$ for an auxiliary function $g$ (not the same function as in the proof of Theorem~\ref{N odd average for K}), which is also multiplicative. By a direct computation with the M\"{o}bius inversion formula, $g$ vanishes unless $d$ is squarefree. Moreover, $g(2)=-\frac{1}{3}$, while for odd primes $\ell$,
\begin{equation}\label{eq:gvalue} g(\ell)=\frac{1}{(\ell-2)(\ell+1)}. \end{equation}
 Writing $\pi(x;d,1)$ for the number of primes $p \leq x$ with $p \equiv 1\mod{d}$, we have
\begin{align}\notag \sum_{p \leq x} F(p-1) &= \sum_{p \leq x} \sum_{d \mid p-1} g(d) \\
	&= \sum_{d \leq (\log{x})^A} g(d) \pi(x;d,1) + \sum_{(\log{x})^A < d \leq x} g(d) \pi(x;d,1).
\label{eq:fpfirst}\end{align}
We first consider the second sum on the right-hand side. Trivially, $\pi(x;d,1) < x/d$, and so
\begin{equation} \label{eq:trivial in AP}
\left|\sum_{(\log{x})^A < d \leq x} g(d) \pi(x;d,1)\right| \leq x \sum_{d > (\log{x})^A} \frac{|g(d)|}{d}.
\end{equation}
When $g(d)$ is nonvanishing, the formula~\eqref{eq:gvalue} yields
\[
d^2 g(d) \ll \prod_{\substack{\ell \mid d,~\ell > 2}} \frac{\ell^2}{\ell^2-\ell-2} \ll \prod_{\ell \mid d}\left(1-\frac{1}{\ell}\right)^{-1} = \frac{d}{\phi(d)},
\]
and hence $g(d) \ll {1/d\phi(d)}$ for all values of $d$. In particular, using the crude lower bound $\phi(d) \gg d^{1/2}$ (compare with the precise \cite[Theorem~2.9, page 55]{MV}), we find that
$g(d) \ll d^{-3/2}$. Thus, equation~\eqref{eq:trivial in AP} gives
\[
\sum_{(\log{x})^A < d \leq x} g(d) \pi(x;d,1)  \ll x\sum_{d > (\log{x})^A} d^{-5/2} \ll x(\log{x})^{-3A/2},
\]
and so equation~\eqref{eq:fpfirst} becomes
\begin{equation} \label{eq:one sum down}
\sum_{p \leq x} F(p-1) = \sum_{d \leq (\log{x})^A} g(d) \pi(x;d,1) + O\big( x(\log{x})^{-3A/2} \big).
\end{equation}

To deal with the remaining sum, we invoke the Siegel--Walfisz theorem \cite[Corollary~11.21, page 381]{MV}. That theorem implies that for a certain absolute constant $c > 0$,
\begin{align}
\notag  \sum_{d \le (\log{x})^A} g(d) \pi(x;d,1) &= \sum_{d \leq (\log{x})^A} g(d)\bigg(\frac{\pi(x)}{\phi(d)} + O_A\big(x \exp(-c\sqrt{\log{x}})\big)\bigg) \\
&=\pi(x)\sum_{d \le (\log{x})^A}\frac{g(d)}{\phi(d)} + O_A\bigg(x \exp(-c\sqrt{\log{x}}) \sum_{d=1}^\infty |g(d)| \bigg) \notag \\
&=\pi(x)\sum_{d=1}^\infty \frac{g(d)}{\phi(d)} \notag \\
&\qquad{}+ O_A\bigg(\pi(x)\sum_{d > (\log{x})^A}\frac{|g(d)|}{\phi(d)} + x \exp(-c\sqrt{\log{x}}) \sum_{d=1}^\infty |g(d)| \bigg). \notag
\end{align}
In the error term, we again use the crude bounds $g(d) \ll d^{-3/2}$ and $\phi(d) \gg d^{1/2}$, obtaining
\begin{equation*}
\sum_{d \le (\log{x})^A} g(d) \pi(x;d,1) = \pi(x)\sum_{d=1}^\infty \frac{g(d)}{\phi(d)} + O_A\big( \pi(x) (\log x)^{-A} + x \exp(-c\sqrt{\log{x}}) \cdot 1 \big),
\end{equation*}
whereupon equation~\eqref{eq:one sum down} becomes
\begin{equation*}
\sum_{p \leq x} F(p-1) = \pi(x)\sum_{d=1}^\infty \frac{g(d)}{\phi(d)} + O_A\big( x(\log x)^{-A} \big).
\end{equation*}
Finally, the constant in this main term is an absolutely convergent sum of a multiplicative function, and hence it can be expressed as the Euler product
\begin{align*}
\sum_{d=1}^{\infty}\frac{g(d)}{\phi(d)} &= \prod_\ell \bigg( 1 + \frac{g(p)}{\phi(p)} + \frac{g(p^2)}{\phi(p^2)} + \cdots \bigg) \\&= \frac{2}{3}\prod_{\ell>2} \left(1 + \frac{1}{(\ell-1)(\ell-2)(\ell+1)}\right) = \frac{2}{3}J,
\end{align*}
by equation~\eqref{eq:gvalue}. This completes the proof of the proposition.
\end{proof}

\begin{proof}[Proof of Theorem \ref{N prime average for K}]
We first claim that the asymptotic formula~\eqref{K or Kstar} for $K^*$ follows easily from the same asymptotic formula for~$K$. Indeed, for each prime $p$, we have $K^{*}(p)= K(p){p/(p-1)} = K(p) + O(K(p)/p)$. Because each local factor in Definition~\ref{K definition} is of the form $1+O(p^{-2})$, we see that $K$ is absolutely bounded. Thus
\[
\sum_{p \leq x} K^{*}(p) = \sum_{p \leq x}K(p) + O\bigg( \sum_{p\le x} \frac1p \bigg) = \sum_{p \leq x}K(p) + O(\log\log{x}),
\]
and so it suffices to establish the asymptotic formula~\eqref{K or Kstar} for $K$.

For each odd prime $p$, the decomposition~\eqref{K decomposition} gives $K(p) = C_2 F(p-1) G(p)$, where $F$ and $G$ are defined in equations~\eqref{F def} and~\eqref{G def}, respectively. Again, all local factors in these definitions are of the form $1+O(p^{-2})$; hence $G(p) = 1+O(1/p^2)$ and $F$ is absolutely bounded. Therefore,
\begin{align*}\sum_{p \le x} K(p) &= \sum_{p \le x} C_2 F(p-1)G(p) \\
&= C_2 \sum_{p \le x} F(p-1) + O\bigg( 1 + \sum_{p \le x} \frac{F(p-1)}{p^2} \bigg) \\
&= C_2 \sum_{p \le x} F(p-1) + O(1),
\end{align*}
and so the desired asymptotic formula~\eqref{K or Kstar} is a direct consequence of Proposition~\ref{prop:F average}.
\end{proof}

\section{The distribution function of $K^*$} \label{distribution section}

The goal of this section is to establish the existence of the distribution function of $K^*(N)$. We do so by bounding the moments of $K^*(N)$:
\begin{equation} \label{muk def}
\mu_k:= \lim_{x\to\infty} \frac{1}{x}\sum_{N \le x} K^{*}(N)^k.
\end{equation}
We describe below how Theorem~\ref{distribution function theorem} follows from Proposition~\ref{prop:moments bounded}.
Before we can bound these moments, however, we must prove that the moments even exist. In Theorem~\ref{N odd average for K} we determined that $\mu_1 = 1$, and the same method of determining $\mu_k$ applies in general.

\begin{prop}\label{prop:moments exist}
For every natural number $k$, the limit~\eqref{muk def} defining $\mu_k$ exists.
\end{prop}

\begin{proof}
Following the proof of
Proposition \ref{prop:oddavg}, we obtain (with minimal changes to the argument) that for each fixed $k$,
\begin{equation}\label{eq:genk} \sum_{\substack{N \le x}} (K_z(N) R_z(N))^k = x \sum_{\sigma \in \Ss} K_z(\sigma)^k R_z(\sigma)^k d_{\sigma} + O_k(x^{3/4}), \end{equation}
where $z = \frac{1}{10}\log{x}$ and $d_{\sigma}$ is defined in equation~\eqref{eq:dsigmadef}. Note that for $N \leq x$, \begin{align*}
\big( K_z(N) R_z(N))^k & {}- (K(N)R(N) \big)^k \\
&\ll_k  \max\big\{ K(N)R(N), K_z(N) R_z(N) \big\}^{k-1} \cdot \big| K(N)R(N)-K_z(N) R_z(N) \big|  \\
&\ll_k (\log\log{x})^{k-1} \cdot \big| K(N)R(N) - K_z(N) R_z(N) \big|
\end{align*}
by the bounds in equation~\eqref{easy upper bounds}; therefore
\begin{align*}
\sum_{N \le x} K^*(N)^k &= \sum_{N \le x} (K_z(N) R_z(N))^k + \bigg( \sum_{N \le x} \big( (K(N) R(N))^k - (K_z(N)R_z(N))^k \big) \bigg) \\
&= \sum_{N \le x} (K_z(N) R_z(N))^k + O_k\bigg( (\log\log{x})^{k-1} \sum_{N \le x} \big| K(N)R(N) - K_z(N) R_z(N) \big| \bigg).
\end{align*}
Using equation~\eqref{eq:genk} in the main term and the estimate~\eqref{glitch fixer} in the error term, we obtain
\begin{align*}
\sum_{N \le x} K^*(N)^k &= x \sum_{\sigma \in \Ss} K_z(\sigma)^k R_z(\sigma)^k d_{\sigma} + O_k(x^{3/4} + (\log\log{x})^{k-1}x/z) \\
&= x \sum_{\sigma \in \Ss} K_z(\sigma)^k R_z(\sigma)^k d_{\sigma} + O_k\bigg(\frac{x}{\log{x}}(\log\log{x})^{k-1}\bigg).
\end{align*}
Dividing both sides by $x$ and passing to the limit, we deduce that
\begin{equation}\label{eq:sumoversigma} \mu_k = \lim_{x\to\infty} \sum_{\sigma \in \Ss}  K_z(\sigma)^k R_z(\sigma)^k d_{\sigma}, \end{equation}
provided that this limit exists.

To compute the sum over $\sigma$ in \eqref{eq:sumoversigma}, we follow the proof of Lemma \ref{lem:identity}; however, the details are somewhat messier. With the four components $\A$, $\B$, $\C$, $\{e_{\ell}\}_{\ell \in \B}$ of $\sigma$ as before, we write down the expansion for $K_z(\sigma)^k R_z(\sigma)^k d_{\sigma}$ analogous to \eqref{eq:bigmess}. This expansion is made up of three pieces, which are products over primes $\ell$ in $\A$, $\B$, and $\C$. The $\B$ product depends additionally on the tuple $\{e_{\ell}\}_{\ell \in \B}$. We sum over all possibilities for $\{e_\ell\}_{\ell \in \B}$ to remove this dependence. After straightforward but uninspiring computations, we find that fixing only $\A$, $\B$, and~$\C$,
\[ \sum_{\sigma} K_z(\sigma)^k R_z(\sigma)^k d_{\sigma} = \left(\prod_{\ell \in \A} P_{\A}(\ell)\right) \left(\prod_{\ell \in \B}P_{\B}(\ell)\right) \left(\prod_{\ell \in \C} P_{\C}(\ell)\right),  \]
where (we suppress the dependence on $k$ in the notation on the left-hand sides)
\begin{align}\notag
	P_{\A}(\ell) &= (1-\tfrac2\ell)^{k+1} (1-\tfrac1\ell)^{-2k},\\
	P_{\B}(\ell) &= \left(1-\frac{1}{\ell}\right)^{1-k} \sum_{d=1}^\infty \frac{1}{\ell^{d}}\left(1-\frac{1}{\ell^{d}(\ell-1)}\right)^k,
	\label{eq:pdefs}\\
	P_{\C}(\ell)&= \frac{1}{\ell}
	\left(1-\frac{1}{(\ell-1)^2 (\ell+1)}\right)^k.\notag
\end{align}
(Note that when $k=1$, these expressions reduce to the expressions in equation~\eqref{early P defs}.)
To compute the sum appearing in \eqref{eq:sumoversigma}, we sum over $\A$, $\B$, and $\C$, keeping in mind that these sets partition the primes in $[2,z]$. We find that
\[ \sum_{\sigma \in \Ss} K_z(\sigma)^k R_z(\sigma)^k d_{\sigma} = \prod_{\ell \leq z} \left(P_{\A}(\ell) + P_{\B}(\ell) + P_{\C}(\ell)\right), \]
and so from equation~\eqref{eq:sumoversigma},
\begin{equation}\label{eq:mukprod} \mu_k = \prod_{\ell}\left(P_{\A}(\ell) + P_{\B}(\ell) + P_{\C}(\ell)\right). \end{equation}
It remains to show that this product converges. From their definitions \eqref{eq:pdefs}, we find that
\begin{align*}
P_{\A}(\ell) &= 1-2/\ell + O_k(1/\ell^2), \\ P_{\B}(\ell) &= 1/\ell + O_k(1/\ell^2), \\ P_{\C}(\ell) &= 1/\ell + O_k(1/\ell^2).\end{align*}
It follows that each term in the product from equation~\eqref{eq:mukprod} is $1 + O(1/\ell^2)$; consequently, that product converges, which completes the proof of the proposition.
\end{proof}

\begin{remarks} For any given $k$, we can explicitly compute $P_\A$, $P_{\B}$, and $P_\C$ and thus write down an exact expression for $\mu_k$ as an infinite product over primes. For example, taking $k=2$, we find that
\[ \mu_2 = \prod_{\ell} \left(1 + \frac{\ell^5-\ell^3-2\ell^2-2\ell-1}{(\ell-1)^4 (\ell+1)^2 (\ell^2+\ell+1)}\right) \approx 1.261605. \]
\end{remarks}

Now that we know these moments $\mu_k$ exist, we proceed to establish an upper bound for them as a function of~$k$. The following result, well known in the theory of probability (see, for example,  \cite[Theorem 3.3.12, page 123]{Dur}), allows us to pass from such an upper bound to the existence of a limiting distribution function.

\begin{lemma}\label{probability lemma} Let $F_1, F_2, \dots$ be a sequence of distribution functions. Suppose that for each positive integer $k$, the limit $\lim_{n\to\infty} \int u^k \, dF_n(u) = \mu_k$ exists. If
	\[ \limsup_{k\to\infty} \frac{\mu_{2k}^{{1/2k}}}{2k} < \infty, \]
then there is a unique distribution function $F$ possessing the $\mu_k$ as its moments, and $F_n$ converges weakly to $F$.
\end{lemma}

We will apply Lemma \ref{probability lemma} with
$$
F_n(u) := \frac{\#\{m\leq n\colon K^{\ast}(m) \leq u\}}{\#\{m \leq n\}},
$$
for which
$$
\lim_{n\to\infty} \int u^k \, dF_n(u) = \lim_{n\to\infty} \frac1{n} \sum_{m\le n} K^{\ast}(m)^k = \mu_k
$$
(so that the uses of $\mu_k$ in equation~\eqref{muk def} and Lemma~\ref{probability lemma} are consistent). In light of Lemma~\ref{probability lemma},  Theorem~\ref{distribution function theorem} is a consequence of the following upper bound.

\begin{prop}\label{prop:moments bounded}
The moments $\mu_k$  defined in equation~\eqref{muk def} satisfy $\log\mu_k \ll k\log\log k$. In particular, $(\mu_{2k}^{1/2k})/2k \ll (\log k)^A/k$ for some constant~$A$.
\end{prop}

\begin{proof} Recall that $R(N)$ denotes the function $N/\phi(N)$. The number $\mu_k$ is the $k$th moment of the function $K(N) R(N)$, and that function is bounded pointwise by $R(N)$. So $\mu_k$ is bounded above by $\mu_k'$, where
\[
\mu_k' := \lim_{x\to\infty} \frac1x \sum_{N\le x} R(N)^k.
\]
Thus, it suffices to establish the estimate $\log \mu_k' \ll k\log\log k$.

By a result known already to Schur (see \cite[page 194]{Sch}; see also \cite[Exercise 14, page 42]{MV}), we have that for each $k$,
\[
\mu_k' = \prod_{p} \bigg(1-\frac{1}{p} + \frac{1}{p}\bigg(1-\frac1p\bigg)^{\!-k}\,\bigg) = \prod_p \bigg( 1 + \frac{1}{p}\bigg( \bigg( \frac p{p-1} \bigg)^{k} - 1^{k} \bigg) \bigg).
\]
By the mean value theorem,
\begin{align*}
1 + \frac{1}{p}\bigg( \bigg( \frac p{p-1} \bigg)^{k} - 1^{k}\bigg) &= 1 + O\bigg( \frac k{p(p-1)} \bigg( \frac p{p-1} \bigg)^{k-1} \bigg) \\
&= 1 + O\bigg( \frac k{p^2} \bigg( 1+\frac1{p-1} \bigg)^{k-1} \bigg) \\
&< 1 + O\bigg( \frac k{p^2} \exp\bigg( \frac{k-1}{p-1} \bigg) \bigg),
\end{align*}
and so
\begin{equation} \label{final mu estimate}
\mu_k' < \prod_{p\le k} \bigg( 1 + O\bigg( \frac k{p^2} \exp\bigg( \frac{k-1}{p-1} \bigg) \bigg) \bigg) \prod_{p> k} \bigg( 1 + O\bigg( \frac k{p^2} \exp\bigg( \frac{k-1}{p-1} \bigg) \bigg) \bigg).
\end{equation}
In the first product, we use the crude inequality
\[
1 + O\bigg( \frac k{p^2} \exp\bigg( \frac{k-1}{p-1} \bigg) \bigg) < 1 + O\bigg( k \exp\bigg( \frac k{p-1} \bigg) \bigg) \ll k \exp\bigg( \frac k{p-1} \bigg),
\]
so that for some absolute constant $C$,
\begin{align*} \notag
\prod_{p\le k} \bigg( 1 + O\bigg( \frac k{p^2} \exp\bigg( \frac{k-1}{p-1} \bigg) \bigg) \bigg) &\le \prod_{p\le k} Ck \exp\bigg( \frac k{p-1} \bigg) \\
&\leq (Ck)^{\pi(k)}\exp\bigg(k \sum_{p\le k} \frac{1}{p-1}\bigg) \\
&= \exp(O(k)) \exp(O(k\log\log{k})).
\end{align*}
In the second product, the exponential factor is uniformly bounded, and so
\begin{align*}
\prod_{p> k} \bigg( 1 + O\bigg( \frac k{p^2} \exp\bigg( \frac{k-1}{p-1} \bigg) \bigg) \bigg) &= \prod_{p>k} \bigg( 1 + O\bigg( \frac k{p^2} \bigg) \bigg) \\
&< \prod_{p> k} \bigg( \exp \bigg( O\bigg( \frac k{p^2} \bigg) \bigg) \bigg) \\
&\le \exp \bigg( O \bigg( \sum_p \frac k{p^2} \bigg) \bigg) = \exp(O(k)).
\end{align*}
In light of these last two estimates, equation~\eqref{final mu estimate} yields $\mu_k' \le \exp(O(k\log\log{k}))$ as required.
\end{proof}

\begin{remarks} It is worthwhile to make a few remarks about the behavior of $D(u)$.
Let $u_0 := \frac{2}{3}C_2$. We can view equation~\eqref{eq:convenienteulerfactorization}, with $z=\infty$, as providing us with a conveniently factored Euler product expansion of $K^{\ast}(N)$. Comparing the terms of this expansion with those in the product expansion for $C_2$, one sees that $K^{\ast}(N) > u_0$ for all $N$. In fact, one finds that $K^{\ast}(N)$ is bounded away from $u_0$ unless all of the small odd primes belong to $\A$, i.e., unless $N(N-1)$ possesses no small odd prime factors. Conversely, if $N(N-1)$ has no small odd prime factors, an averaging argument shows that $K^{\ast}(N)$ is usually close to $u_0$. In this way, one proves that $D(u_0)=0$ while $D(u) > 0$ for $u > u_0$.

Since $K(N)$ is absolutely bounded and bounded away from zero, several results on $D(u)$ follow immediately from corresponding results for the distribution function of $N/\phi(N)$, whose behavior has been studied by Erd\H{o}s \cite{Erd46} and Weingartner \cite{Wei07, Wei12}. In particular, from \cite[Theorem 1]{Erd46}, we see that $D(u) > 1-\exp(-\exp(Cu))$ for a certain constant $C>0$ and all large $u$.
	
Finally, we remark that there is an alternative, more arithmetic approach to the proof of Theorem \ref{distribution function theorem}, based on ideas and results of Erd\H{o}s \cite{Erd35D} and Shapiro \cite{Sha}. This approach allows us to show that the distribution function $D(u)$ of Theorem \ref{distribution function theorem} is continuous everywhere and strictly increasing for $u > u_0$. We omit the somewhat lengthy arguments for these claims.
\end{remarks}

\section*{Acknowledgements}
We thank Igor Shparlinski for bringing the reference \cite{BS} to our attention. We also thank the anonymous referee for a careful reading of the manuscript.

\end{document}